\documentclass[11pt]
{amsart}
\usepackage{amssymb,amsmath,amsthm,amsfonts,amsopn,url,color,hyperref,%enumerate,
mathtools,microtype,MnSymbol,dutchcal,
mathrsfs,relsize}
\usepackage{scalerel}
\usepackage{stackengine,wasysym}
\usepackage[shortlabels]{enumitem}

\usepackage[normalem]{ulem}
\usepackage[all]{xy}
\input xy
\xyoption{all}
\usepackage{amscd}
\usepackage{soul}

\theoremstyle{plain}
\newtheorem{thm}{Theorem}[section]

\newtheorem{prop}[thm]{Proposition}
\newtheorem{lemma}[thm]{Lemma}
\newtheorem{cor}[thm]{Corollary}

\theoremstyle{definition}
\newtheorem{defn}[thm]{Definition}
\newtheorem*{defn*}{Definition}
\newtheorem*{question*}{Question}

\newtheorem{example}[thm]{Example}
\newtheorem*{example*}{Example}
\newtheorem{rem}[thm]{Remark}
\newtheorem*{rem*}{Remark}

\newcommand{\field}[1]{\mathbb{#1}}
\newcommand{\N}{\field{N}}
\newcommand{\Z}{\field{Z}}
\newcommand{\Q}{\field{Q}}

\newcommand{\ideal}[1]{\mathfrak{#1}}

\newcommand{\p}{\ideal{p}}

\newcommand{\ra}{\rightarrow}

\DeclareMathOperator{\ann}{ann}

\newcommand{\be}{\begin{enumerate}}
\newcommand{\ee}{\end{enumerate}}

\newcommand{\li}%{{\mathrm{lic}}}
 {\leftfootline}

\renewcommand{\phi}{\varphi}

\let\int\relax
\DeclareMathOperator{\int}{i}

\author{Neil Epstein}
\address{Department of Mathematical Sciences \\ George Mason University \\ Fairfax, VA  22030}
\email{nepstei2@gmu.edu}

\title[$W$-Egyptian rings]{Rings that are Egyptian with respect to a multiplicative set}
\date{September 19, 2023}
\begin{document}
\maketitle
\begin{abstract}
The notion of an \emph{Egyptian} integral domain $D$ (where every fraction can be written as a sum of unit fractions with denominators from $D$) is extended here to the notion that a ring $R$ is \emph{$W$-Egyptian}, with $W$ a multiplicative set in $R$.  The new notion allows denominators just from $W$.  It is shown that several results about Egyptian domains can be extended to the $W$-Egyptian context, though some cannot, as shown in counterexamples. In particular, being a sum of unit fractions from $W$ is \emph{not} equivalent to being a distinct sum of unit fractions from $W$, so we need to add the notion of the \emph{strictly} $W$-Egyptian ring.  Connections are made with Jacobson radicals, power series, products of rings, factor rings, modular arithmetic, and monoid algebras.
\end{abstract}

\section{Introduction}
In ancient Egypt, rational numbers were typically represented as sums of reciprocals of distinct integers -- for instance $\frac 29$ might be thought of either as $\frac 16+\frac 1{18}$ or as $\frac 15 + \frac 1{45}$.  In modern terms, the problem was to represent a fraction as a sum of distinct \emph{unit fractions}.  In medieval times, none other than Fibonacci showed that this is always possible \cite{DuGr-FibE}.  See \cite{Gil-phar, Rei-count} for two excellent and supremely accessible book-length introductions to the way ancient Egyptians did mathematics.

In the twentieth century, the study of so-called Egyptian fractions was taken up again.  Many interesting questions remain, the most famous of which are the Erd\H os-Straus and Sierpi\'nski conjectures, both from mid-century, which ask whether for every $n>4$, the number $\frac 4n$ (resp.  $\frac 5n$) can be represented as a sum of \emph{three} distinct unit fractions.  For a survey of these and other related developments, see \cite{Gra-ErdEgypt}.

A natural question from a ring-theoretic point of view is then: Given an integral domain $D$, can one represent any element of the fraction field as a sum of reciprocals of distinct elements of $D$?  The study of such domains was initiated in \cite{GLO-Egypt}.  In that paper, an element $0 \neq d \in D$ was called \emph{Egyptian} if it could be so represented, and $D$ was called an \emph{Egyptian domain} if all nonzero elements were Egyptian.  There it was proved that $\mathbb Z$ is Egyptian, that group rings over a field or $\Z$ are Egyptian, that $k[x]$ is not Egyptian but all its overrings are, that overrings and integral extensions of Egyptian domains are Egyptian, that any domain with a nonzero Jacobson radical is Egyptian, and that any element that can be written as a sum of reciprocals of elements of $D$ can be rewritten as a sum of reciprocals of \emph{distinct} elements of $D$. 
The current author entered the story at this point, as I continued the study of Egyptian domains in my article \cite{nme-Edom}. I showed, e.g., that no positively graded domain is Egyptian, that domains which are affine semigroup rings are not Egyptian unless the semigroup is a group, that $W^{-1}D[X]$ is Egyptian whenever $D$ is Egyptian and $W$ contains positive degree elements, and that the pullback of an Egyptian domain under the usual pullback construction is Egyptian. It follows that an excellent Noetherian domain is Egyptian if and only if its integral closure is.  %I also introduced the notions of \emph{locally} and \emph{generically} Egyptian domains, showing that any domain that is a finitely generated algebra over a field or $\Z$ is locally Egyptian, even though by \cite[Proposition 1]{GLO-Egypt} such algebras are often not Egyptian.

The next logical question to ask was whether or how much of this theory can be ported to rings with zero-divisors.  See \cite[Problem 5]{GLO-Egypt}.  The current paper provides a framework for answering this question.  Namely, given a ring $R$ and a multiplicative set $W$, we will explore what it means for an element $g \in W^{-1}R$ to be \emph{(strictly) $W$-Egyptian}.  Using this framework, I prove generalizations of various results from \cite{nme-Edom} and \cite{GLO-Egypt}.  However, I also show that not \emph{all} results from those papers generalize, necessitating a distinction marked by the word ``strictly''.

For the reader's convenience, I provide here a list of some results from \cite{nme-Edom} and \cite{GLO-Egypt}, most of which have been alluded to above, which will be dealt with relative to a multiplicative set in the rest of the paper:

\begin{thm}\label{thm:prevs}
Let $D$ be an integral domain and $0 \neq d \in D$.
\begin{enumerate}
    \item\label{it:nonstrict}\cite[Theorem 2]{GLO-Egypt} If $d$ can be expressed as a sum of unit fractions with denominators in $D$, then it can be rewritten as a sum of \emph{distinct} such unit fractions.  Hence, if every $d\in D$ can be written as a sum of unit fractions from $D$, then $D$ is an Egyptian domain.
    \item\label{it:Jacrad}\cite[Example 3]{GLO-Egypt} If the Jacobson radical of $D$ is nonzero, then $D$ is Egyptian.  Indeed $d$ can be written as a sum of reciprocals of at most three distinct elements of $D$.
    \item\label{it:ps}\cite[Proposition 2]{GLO-Egypt} The power series ring $D[\![X]\!]$ is Egyptian.
    \item\label{it:generic}\cite[Lemma 3.9]{nme-Edom} Let $D \subseteq R$ be an extension of integral domains such that $D$ is Egyptian and $R$ can be finitely generated as an $D$-algebra, say $R = D[u_1, \ldots, u_t]$.  Then $R[1/u]$ is Egyptian, where $u = \prod_{j=1}^t u_j$.  In particular, any domain that is finitely generated over a field or over $\Z$ can be made Egyptian by inverting one element.
    \item\label{it:semigroup}\cite[Theorem 2.6]{nme-Edom} Let $n$ be a positive integer, and let $\Lambda$ be an additive submonoid of $\Q^n$ that is not a group. Then $D[\Lambda]$ is \emph{not} Egyptian.
\end{enumerate}
\end{thm}

\section{$W$-Egyptian rings}
\begin{defn}
Let $R$ be a commutative ring and $W$ a multiplicative set.  Let $c \in W^{-1}R$.  We say that $c$ is \emph{$W$-Egyptian} (resp. \emph{strictly $W$-Egyptian}) if it can be written as the sum of reciprocals of (resp. distinct) elements of $W$.  If $b \in R$, we say $b$ is (\emph{strictly})\emph{ $W$-Egyptian} if $b/1 \in W^{-1}R$ is (strictly) $W$-Egyptian.

We say that $R$ itself is \emph{$W$-Egyptian} (resp. \emph{strictly $W$-Egyptian}) if every nonzero element of $R$ satisfies the corresponding property.

If $W = $ the set of regular elements of $R$, we omit the prefix and speak simply of Egyptian or strictly Egyptian elements or rings.
\end{defn}

Note that when $R$ is an integral domain and $0\neq r \in R$, the above agrees with the terminology of calling $R$ or $r$ \emph{Egyptian} as introduced in \cite{GLO-Egypt}.

\begin{example}\label{ex:nonstrict}
Unlike integral domains (see Theorem~\ref{thm:prevs}(\ref{it:nonstrict})), the Egyptian and strictly Egyptian properties are \emph{not} in general equivalent for rings with zero-divisors.  Consider the ring $R=\Z/6\Z$.  The regular elements are 1 and 5 (each of which is its own reciprocal), so there can be at most 4 sums of distinct such elements (including the empty sum).  In particular, the element $2$ is not strictly Egyptian, so the ring $\Z/6\Z$ is not strictly Egyptian.  However, $\Z/6\Z$ is Egyptian, as every element is a positive integer multiple of $1/1$.  For the same reason, every quotient of $\Z$ is Egyptian.

A similar analysis shows that $\Z/4\Z$ is not strictly Egyptian.

By contrast, $\Z/8\Z$ \emph{is} strictly Egyptian.  Indeed, every element is a sum of at most two distinct units, as the set of units is $\{1,3,5,7\}$ and we have $3+7=2$, $1+3=4$, $1+5=6$, and $1+7=0$.
\end{example}

\begin{rem}
If $0 \in W$, then $W^{-1}R=0$ (see \cite[p. 38, Example 2]{AtMac-ICA}), so $R$ is strictly $W$-Egyptian as $\frac r1 =0= \frac 10$ in $W^{-1}R$ for all $r\in R$.
\end{rem}

\begin{rem}
Let $R$ be a ring, $W$ a multiplicative set, and $f = r/w \in W^{-1}R$.  If $r/1$ is $W$-Egyptian, then so is $f$.  To see this, write $\frac r1 = \frac 1 {w_1} + \cdots + \frac 1 {w_n}$ with each $w_j \in W$.  Then $f=\frac rw = \frac 1 {w_1w} + \cdots + \frac 1 {w_nw}$. 
\end{rem}

\begin{rem}\label{rem:mscontain}
Let $R$ be a ring and $V \subseteq W$ multiplicatively closed subsets of $R$. It follows from the definition that if $R$ is $V$-Egyptian (resp. strictly $V$-Egyptian), it is also $W$-Egyptian (resp. strictly $W$-Egyptian).
\end{rem}

The following two results are generalizations of Theorem~\ref{thm:prevs}(\ref{it:Jacrad}).%$\cite[Example 3]{GLO-Egypt}.

\begin{prop}
Let $R$ be a commutative ring.  Let $j$ be an element of the Jacobson radical of $R$.  Let $W$ be the multiplicative set generated by $j$ and the units of $R$.  Then $R$ is $W$-Egyptian. 
\end{prop}

\begin{proof}
Let $0 \neq f \in R$.  Then $jf-1$ is a unit in $R$, so $(jf-1)^{-1} \in W$.  Then in $W^{-1}R$, \[
\frac f1 = \frac {1+(jf-1)}j = \frac 1j + \frac 1 {j(jf-1)^{-1}}. \qedhere
\]
\end{proof}

\begin{thm}\label{thm:jrreg}
If the Jacobson radical of a ring $R$ admits a regular element $j$, then $R$ is strictly $W$-Egyptian (hence also strictly Egyptian by Remark~\ref{rem:mscontain}), where $W$ is the multiplicative set generated by $j$ and the units of $R$.  In fact, any nonzero element of $R$ can be written as a sum of at most \emph{two} distinct unit fractions from $W$.
\end{thm}

\begin{proof}
Let $0 \neq f \in R$. As in the previous proof, we have $\frac f1 = \frac 1j + \frac 1 {j(jf-1)^{-1}}$.  This is already a strictly $W$-Egyptian representation unless $j=j(jf-1)^{-1}$ as elements of $R$.  Since $j$ is regular, the above is equivalent to saying $1=(jf-1)^{-1}$, which in turn is equivalent to $jf=2$.  Now, $j^2$ is also a regular element in the Jacobson radical, so $\frac f1 = \frac 1{j^2} + \frac 1 {j^2(j^2f-1)^{-1}}$ is a strictly $W$-Egyptian representation unless $j^2f=2$.  But since $jf=2$ as well, we have $jf(j-1)=j^2f-jf =2-2=0$, so since $j$ is regular and $j-1$ is a unit it follows that $f=0$, which contradicts the original assumption.
\end{proof}

\begin{rem}
    Theorem~\ref{thm:jrreg} is an improvement of Theorem~\ref{thm:prevs}(\ref{it:Jacrad}) even in its original context, as we now know that the upper bound of \emph{three} summands given there can be reduced to \emph{two}.
\end{rem}

The following is a generalization of Theorem~\ref{thm:prevs}(\ref{it:ps}).%\cite[Proposition 2]{GLO-Egypt}.
\begin{cor}
For any nonzero ring $R$, its power series extension $R[\![X]\!]$ is strictly Egyptian.
\end{cor}

\begin{proof}
$X$ is a regular element in the Jacobson radical of $R[\![X]\!]$.
\end{proof}

The next two results track how Egyptianness with respect to a multiplicative set passes from generating sets and across injective ring maps.

\begin{prop}\label{pr:gens}
Let $R$ be a commutative ring and $W$ a multiplicative set.  Let $G$ be a generating set for $R$ as a ring.  If every element of $G$ is $W$-Egyptian, then $R$ is $W$-Egyptian.
\end{prop}

\begin{proof}
Suppose $r,s\in R$ are $W$-Egyptian.  Write \[
\frac r1 = \mathlarger{\mathlarger{\sum}}_{i=1}^n \frac 1 {w_i} \quad \text{and} \quad \frac s1 = \mathlarger{\mathlarger{\sum}}_{j=1}^m \frac 1{v_j},
\]
with $m,n\in \N$, and $w_i,v_j \in W$.  Then \[
\frac{r+s}1= \frac r1 + \frac s1 = 
\mathlarger{\mathlarger{\sum}}_{i=1}^n \frac 1{w_i} + \mathlarger{\mathlarger{\sum}}_{j=1}^m \frac 1{v_j} \quad \text{and} \quad \frac{rs}1 = \frac r1 \cdot \frac s1=\mathlarger{\mathlarger{\sum}}_{i=1}^n \mathlarger{\mathlarger{\sum}}_{j=1}^m \frac 1 {w_i v_j}.
\]
Since $W$ is a multiplicative set, the above are $W$-Egyptian representations of $r+s$ and $rs$.  Since every element of $R$ is built from $G$ by sums and products of elements, the result follows by induction.
\end{proof}

\begin{prop}
Let $j: A \rightarrow R$ be an injective map of commutative rings.  Let $W$ be a multiplicative set in $A$.  Suppose $R$ is $j(W)$-Egyptian.  Then there is a ring map $i: R \rightarrow W^{-1}A$ that fits into a commutative diagram: \[\xymatrix{
A \ar[rr]^j \ar[d]_{\ell_A}&&R \ar[dll]_i \ar[d]^{\ell_R}\\
W^{-1}A \ar[rr]_{W^{-1}j}&&j(W)^{-1}R
}\]
where $\ell_A: A \ra W^{-1}A$ and $\ell_R: R \ra j(W)^{-1}R$ are the localization maps $a\mapsto \frac a1$ and $r \mapsto \frac r1$.  Moreover, $A$ is $W$-Egyptian.
\end{prop}

In particular if $A$ is a domain and $0\notin W$, it follows that any $W$-Egyptian domain that contains $A$ as a subring must be isomorphic to an overring of $A$.

\begin{proof}
Define $i$ as follows. For $r\in R$, since $R$ is $j(W)$-Egyptian there exist $w_1, \ldots, w_n \in W$ such that in $W^{-1}R$, $\frac r1 = \frac 1{j(w_1)} + \frac 1{j(w_2)} + \cdots + \frac 1{j(w_n)}$.  We set $i(r) := \frac 1{w_1} + \cdots + \frac 1{w_n}$.

First we must show $i$ is well-defined.  Accordingly, suppose $v_1, \ldots, v_s \in W$ such that $\frac r1 = \frac 1 {j(v_1)} + \cdots + \frac 1{j(v_s)}$.  Then we have $(W^{-1}j)(\frac 1 {w_1} + \cdots +\frac 1{w_n}) = (W^{-1}j)(\frac 1 {v_1} + \cdots + \frac 1{v_s})$.  But since $j$ is injective and localization is an exact functor, the map $W^{-1}j :W^{-1}A \ra W^{-1}R$ is also injective.  Hence, $\frac 1 {w_1} + \cdots +\frac 1{w_n} = \frac 1 {v_1} + \cdots + \frac 1{v_s}$, completing the proof that $i$ is well-defined as a set map.

After this, it is easy to see that $i$ is in fact a ring map.

Now let $a\in A$.  Again since $R$ is $j(W)$-Egyptian, there exist $w_1, \ldots, w_n \in W$ such that $\frac {j(a)}1 = \frac 1{j(w_1)} + \cdots + \frac 1{j(w_n)}$.  That is, $\frac a1 - \frac 1{w_1} - \cdots - \frac 1{w_n} \in \ker (W^{-1}j) = 0$, so that $\ell_A(a) = \frac a1 = \frac 1{w_1} + \cdots + \frac 1{w_n} = i(j(a))$. Hence, $i \circ j = \ell_A$, and $A$ is $W$-Egyptian.

To show that the lower triangle is commutative, let $r\in R$. Choose $w_1, \ldots, w_n \in W$ such that $\frac r1 = \frac 1{j(w_1)} + \cdots + \frac 1{j(w_n)}$.  Then $\ell_R(r) = \frac r1= (W^{-1}j)(\frac 1{w_1} + \cdots +\frac 1{w_n}) = ((W^{-1}j) \circ i)(r)$.
\end{proof}

The following is a generalization of Theorem~\ref{thm:prevs}(\ref{it:generic}).%\cite[Lemma 3.9]{nme-Edom}.

\begin{thm}
    Let $A$ be a $W$-Egyptian ring ($W$ a multiplicative subset of $A$), $R$ an $A$-algebra finitely generated by $u_1, \ldots, u_t \in R$, and $u= \prod_{j=1}^t u_j$. Then $R_u$ is $V$-Egyptian, where $V:=\langle \ell(\phi(W)),u,  u_1^{-1}, \ldots, u_t^{-1}\rangle$, $\phi: A \ra R$ is the structure map, and $\ell: R \ra R_u$ is the localization map.
\end{thm}

\begin{proof}
First note that for $1\leq i \leq t$, $u_i^{-1} \in R_u$, as $u_i^{-1} = \frac {\prod_{j\neq i} u_j} u$.  Thus, $V \subseteq R_u$.  Note that $R_u$ is generated as a ring by $\ell(R)$ and $u^{-1}$.  We have $u^{-1} = \frac 1u$, where $u\in V$, so by Proposition~\ref{pr:gens} we need only show that the elements of $\ell(R)$ are $V$-Egyptian.

Let $r\in R$.  Then there is some polynomial $p \in \phi(A)[X_1, \ldots, X_t]$ such that $r = p(u_1, \ldots, u_t)$.  This $p$ is a sum of terms of the form $\phi(a) X_1^{k_1} \cdots X_t^{k_t}$, with $a\in A$ and $k_1, \ldots, k_t \in \N_0$.  But $a$ is $W$-Egyptian, so we have in $W^{-1}A$ that $\frac a1 = \frac 1 {w_1} + \cdots + \frac 1{w_n}$, with $n\in \N$ and each $w_j \in W$.  Thus in $V^{-1}(R_u)$, \[
\frac {\phi(a) u_1^{k_1} \cdots u_t^{k_t}}1 = \mathlarger{\mathlarger{\sum}}_{j=1}^t \frac 1 {\phi(w_j) (u_1^{-1})^{k_1} \cdots {(u_t^{-1})^{k_t}}}
\]
Since all of the above are $V$-unit fractions in $V^{-1}R_u$, and since $r$ is a sum of finitely many such terms, it follows that $\frac {\ell(r)}1 \in V^{-1}R_u$ is $V$-Egyptian.

Finally, any element of $R_u$ is of the form $\frac r{u^s}$ for some $s \in \N$.  By the above, there exist $v_1, \ldots, v_m \in V$ with $\frac r1 = \displaystyle\sum_{j=1}^m \frac 1{v_j}$.  Since $u^s v_j \in V$ for $1\leq j \leq m$, it follows that $\frac{r/u^s}1 = \frac r{u^s} = \displaystyle \sum_{j=1}^m \frac 1{u^sv_j}$ is $V$-Egyptian.
\end{proof}

\begin{example}
This more flexible version of Egyptianness works well with finite direct products of rings.  Indeed, if $R$ is $V$-Egyptian and $S$ is $W$-Egyptian, then $R \times S$ is $(V \times W)$-Egyptian.

To see this, let $r\in R$, $s \in S$, and write $r = \sum_{i=1}^n 1/v_i$ and $s = \sum_{j=1}^m 1/w_j$.  Then in $(R \times S)_{V \times W}$, we have \[
\frac {(r,s)}{(1,1)} = \frac{(r,1)}{(1,1)} \cdot \frac{(1,s)}{(1,1)}= \left(\sum_{i=1}^n \frac{(1,1)}{(v_i,1)}\right) \cdot \left(\sum_{j=1}^m \frac{(1,1)}{(1,w_j)} \right) = \sum_{i=1}^n \sum_{j=1}^m \frac{(1,1)} {(v_i,w_j)}.
\]
\end{example}

\begin{example}
However, the above does not extend to infinite direct products.  Let $R = \prod_{j=1}^\infty \Z$, and consider the element $e = (1,2,3,4,\ldots)\in R$.  If $e$ is Egyptian, then there are elements $c_1, \ldots, c_t \in \prod_{j=1}^\infty (\Z \setminus \{0\})$ with $e = \frac 1{c_1} + \cdots \frac 1{c_t}$.  Write $c_i=(c_{i1},c_{i2},\ldots)$.  Then matching the $t+1$st coordinate of each side, we have \[
t+1 = \frac 1 {c_{1,t+1}} + \frac 1 {c_{2,t+1}} + \cdots + \frac 1 {c_{t,t+1}}.
\]
But since each $c_{ij}$ is a nonzero integer, we have each $1/c_{ij} \leq 1$.  Thus, $t+1 \leq 1+ 1+ \cdots + 1 = t$, which is a contradiction.
\end{example}

Finally, after a lemma that shows that relative Egyptianness passes to quotient rings, we have a theorem that says that semigroup rings are almost never Egyptian.

\begin{lemma}\label{lem:quotient}
Let $R$ be a commutative ring, $I$ an ideal of $R$, and $W \subseteq R$ a multiplicative set.  Let $r\in R$ be $W$-Egyptian.  Then $\bar r \in \bar R$ is $\bar W$-Egyptian, where $\ \bar{ }\ $ denotes going modulo $I$.  Thus, if $R$ is $W$-Egyptian, then $\bar R$ is $\bar W$-Egyptian.
\end{lemma}

\begin{proof}
Write $\frac r 1 = \frac 1 {w_1} + \cdots + \frac 1 {w_n}$, with each $w_i \in W$.  Then there is some $v\in W$ such that in $R$, we have $\left(\prod_{i=1}^n w_i\right) vr = v\cdot \sum_{i=1}^n \prod_{j \neq i} w_j$.  Going mod $I$, we have $\left(\prod_{i=1}^n \overline{w_i}\right) \bar v\bar r = \bar v\cdot \sum_{i=1}^n \prod_{j \neq i} \overline{w_j}$, with $\bar v$ and each $\overline{w_i}$ in $\bar W$.  Then in ${\bar W}^{-1}\bar R$, we have $\frac {\bar r}1 = \frac 1{\overline{w_1}} + \cdots + \frac 1 {\overline{w_n}}$, whence $\bar r$ is $\bar W$-Egyptian.
\end{proof}

The following is a generalization of Theorem~\ref{thm:prevs}(\ref{it:semigroup}).%\cite[Theorem 2.6]{nme-Edom}.
\begin{thm}
Let $A$ be a commutative ring.  Let $\Lambda \subseteq \Q^n$ be an additive submonoid, and $R = A[\Lambda]$.  If $\Lambda$ is not a group, then $R$ is not an Egyptian ring.
\end{thm}

\begin{proof}
Let $\p$ be a minimal prime of $A$.  Then $\p R$ consists of zero-divisors.  To see this, let $f\in \p R$.  Then $f = p_1 \lambda_1 + \cdots + p_t \lambda_t$, with $p_i \in \p$ and $\lambda_i \in \Lambda$.  Since $\p$ is a weak Bourbaki prime of $A$ (see \cite[Definition (e)]{IrRu-ass}), there is some nonzero $d\in A$ with $(p_1, \ldots, p_t) \subseteq \ann_Ad$ \cite[Chapitre II, Lemme 1.13]{Laz-plat}.
Thus, $df = (dp_1)\lambda_1 + \cdots + (dp_t)\lambda_t = 0$, but $d\neq 0$, so $f$ is a zero-divisor.

Now let $W$ be the set of regular elements of $R$. Then by the above, $W \cap \p R = \emptyset$, so $\bar W$ consists of nonzero elements of $R/\p R$.  By \cite[Theorem 2.6]{nme-Edom}, $R/\p R \cong (A/\p)[\Lambda]$ is a non-Egyptian integral domain, hence it cannot be $\bar W$-Egyptian either, since $\bar W$ consists of nonzero elements; see Remark~\ref{rem:mscontain}.  Then by Lemma~\ref{lem:quotient}, $R$ is not Egyptian.
\end{proof}

\section*{Acknowledgment}
I offer here my thanks to the anonymous referee, who insisted that I provide explanations herein to make the paper more self-contained. As a result, I believe the paper reads better now.  The referee also gave several useful corrections and suggestions which also have improved the paper.

\providecommand{\bysame}{\leavevmode\hbox to3em{\hrulefill}\thinspace}
\providecommand{\MR}{\relax\ifhmode\unskip\space\fi MR }
\providecommand{\MRhref}[2]{%
	\href{http://www.ams.org/mathscinet-getitem?mr=#1}{#2}
}
\providecommand{\href}[2]{#2}

\end{document}